\documentclass[12pt]{slides}
\usepackage{amsmath,amsfonts,amssymb,amsthm}
\begin{document}
\author{\textbf{Denis~I.~Saveliev}}
\title{\textbf{{\LARGE Common idempotents in compact left topological left semirings}}} 
\date{24 January 2010, Muscat\\ ICAA \vfil }
\maketitle

\theoremstyle{plain}
\newtheorem{thm}{Theorem}
\newtheorem*{prb}{Problem}
\newtheorem*{q}{Question}
\newtheorem{coro}{Corollary}
\newtheorem*{tm}{Theorem}
\newtheorem*{pr}{Proposition}
\newtheorem*{cm}{Claim}
\newtheorem*{cor}{Corollary}
\newtheorem*{cnj}{Conjecture}
\newtheorem*{lm}{Lemma}
\newtheorem*{fct}{Fact}
\newtheorem*{fcts}{Facts}

\theoremstyle{definition}
\newtheorem*{df}{Definition}
\newtheorem*{rfr}{References}
\newtheorem*{empt}{}

\theoremstyle{remark}
\newtheorem*{rmk}{Remark}
\newtheorem*{exm}{Example}
\newtheorem*{exms}{Examples}

\newcommand{\uhr}{\!\upharpoonright\!}

\newcommand{\germ}{\mathfrak}

\newcommand{\scc}{\mathop {\text{\it {\ss}}\!\!\;}\nolimits }

\newcommand{\smf}{\mathop {^\smallfrown}\nolimits }
\newcommand{\eqs}{ {\;=^{\!*}\,} }
\newcommand{\lesss}{ {\;<^{\!*}\,} }
\newcommand{\leqs}{ {\;\le^{\!*}\,} }
\newcommand{\image}{\/``\,}
\newcommand{\wo}{\mathop {\mathrm {wo\,}}\nolimits }
\newcommand{\pwo}{\mathop {\mathrm {pwo\,}}\nolimits }
\newcommand{\wf}{\mathop {\mathrm {wf\,}}\nolimits }
\newcommand{\ewf}{\mathop {\mathrm {ewf\,}}\nolimits }
\newcommand{\cf}{ {\mathop{\mathrm {cf\,}}\nolimits} }
\newcommand{\hcf}{ {\mathop{\mathrm {hercf\,}}\nolimits} }
\newcommand{\cl}{ {\mathop{\mathrm {cl\,}}\nolimits} }
\newcommand{\tc}{ {\mathop{\mathrm {tc\,}}\nolimits} }
\newcommand{\dom}{ {\mathop{\mathrm {dom\,}}\nolimits} }
\newcommand{\ran}{ {\mathop{\mathrm {ran\,}}\nolimits} }
\newcommand{\fld}{ {\mathop{\mathrm {fld\,}}\nolimits} }
\newcommand{\ot}{ {\mathop{\mathrm {ot\,}}\nolimits} }
\newcommand{\otp}{ {\mathop{\mathrm {otp\,}}\nolimits} }
\newcommand{\lh}{ {\mathop{\mathrm {lh\,}}\nolimits} }
\newcommand{\ext}{\mathrm{ext}}
\newcommand{\rk}{ {\mathrm{rk}} }
\newcommand{\A}{ {\mathrm A} }
\newcommand{\Z}{ {\mathrm Z} }
\newcommand{\ZF}{ {\mathrm {ZF}} }
\newcommand{\ZFA}{ {\mathrm {ZFA}} }
\newcommand{\ZFC}{ {\mathrm {ZFC}} }
\newcommand{\E}{ {\mathrm {AE}} }
\newcommand{\AR}{ {\mathrm {AR}} }
\newcommand{\GR}{ {\mathrm {GR}} }
\newcommand{\WR}{ {\mathrm {WR}} }
\newcommand{\WO}{ {\mathrm {WO}} }
\newcommand{\GWO}{ {\mathrm{GWO}} }
\newcommand{\LO}{ {\mathrm {LO}} }
\newcommand{\AF}{ {\mathrm {AF}} }
\newcommand{\AC}{ {\mathrm {AC}} }
\newcommand{\DC}{ {\mathrm {DC}} }
\newcommand{\GC}{ {\mathrm {GC}} }
\newcommand{\ACM}{ {\mathrm {ACM}} }
\newcommand{\GCM}{ {\mathrm {GCM}} }
\newcommand{\AD}{ {\mathrm {AD}} }
\newcommand{\AInf}{ {\mathrm {AInf}} }
\newcommand{\AU}{ {\mathrm {AU}} }
\newcommand{\AP}{ {\mathrm {AP}} }
\newcommand{\CH}{ {\mathrm {CH}} }
\newcommand{\APr}{ {\mathrm {APr}} }
\newcommand{\ASp}{ {\mathrm {ASp}} }
\newcommand{\ARp}{ {\mathrm {ARp}} }
\newcommand{\AFA}{ {\mathrm {AFA}} }
\newcommand{\BAFA}{ {\mathrm {BAFA}} }
\newcommand{\FAFA}{ {\mathrm {FAFA}} }
\newcommand{\SAFA}{ {\mathrm {SAFA}} }
\newcommand{\NS}{ {\mathrm{NS}} }
\newcommand{\Cov}{ {\mathrm{Cov}} }
\newcommand{\Con}{ {\mathrm{Con}} }
\newcommand{\PA}{ {\mathrm{PA}} }
\newcommand{\BF}{ {\mathrm{BF}} }
\newcommand{\RP}{ {\mathrm{RP}} }
\newcommand{\RC}{ {\mathrm{RC}} }
\newcommand{\RE}{ {\mathrm{RE}} }
\newcommand{\J}{ {\mathrm {J}} }
\newcommand{\Ja}{ {\mathrm {Ja}} }
\newcommand{\Jb}{ {\mathrm {Jb}} }
\newcommand{\Jc}{ {\mathrm {Jc}} }
\newcommand{\Sc}{ {\mathrm {Sc}} }
\newcommand{\wSc}{ {\mathrm {wSc}} }
\newcommand{\Sk}{ {\mathrm {Sk}} }
\newcommand{\GSc}{ {\mathrm {GSc}} }
\newcommand{\GwSc}{ {\mathrm {GwSc}} }
\newcommand{\GSk}{ {\mathrm {GSk}} }
\newcommand{\T}{ {\mathrm {T}} }
\newcommand{\TST}{ {\mathrm {TST}} }
\newcommand{\PI}{ {\mathrm {PI}} }
\newcommand{\LF}{ {\mathrm {LF}} }
\newcommand{\LR}{ {\mathrm {LR}} }
\newcommand{\RK}{ {\mathrm{RK}} }

\newpage

A~classical result of topological algebra states that
any compact left topological semigroup has an idempotent.
This result, in its final form due to Ellis, became crucial 
for numerous applications 
in number theory, algebra, topological dynamics, and ergodic theory.

In this talk, I~consider more complex structures than semigroups: 
{\it left semirings\/}, the algebras with two associative operations 
one of which is also left distributive w.r.t. another one.
By Ellis' result, compact left topological left semirings 
have additive idempotents as well as multiplicative ones.
I~show that they have, moreover, {\it common\/}, i.e., additive and 
multiplicative simultaneously, idempotents.

As an application, I~partially answer a~question related to
algebraic properties of the Stone--\v{C}ech compactification 
of natural numbers. 
Finally, I show that similar arguments establish the existence of 
common idempotents in far more general structures than 
left semirings.

\newpage

\newpage

{\centerline{\Large{\bf One operation:\/}}} 

\noindent
{\centerline{\Large{\bf Ellis' result and applications\/}}}

\newpage

Here I~recall Ellis' theorem on semigroups and discuss 
its importance for various applications.

\newpage

{\it Terminology.\/}\quad
A~{\it groupoid\/} is a~set with a~binary operation.
The usual notation is $(X,\cdot)$ or its variants, like $(X,+)$, or simply~$X$.
In the multiplicative notation, the symbol~$\cdot$ is often omitted.
When the operation satisfies the {\it associativity\/} law
$$(xy)z=x(yz),$$
the groupoid is called a~{\it semigroup\/}.

A~groupoid is {\it left topological\/} iff for any fixed first argument~$a$
the left translation $$x\mapsto ax$$ is continuous.
{\it Right topological\/} groupoids are defined dually.
A~groupoid is {\it semitopological\/} iff it is left and right topological simultaneously,
and {\it topological\/} iff its operation is continuous.
It is not difficult to verify that this hierarchy is not degenerate, even for semigroups.

We shall be interested in left topological groupoids.

\newpage

R.~Ellis in his ``Lectures on topological dynamics" (1969)
published the following simple but remarkable result
(see~[1]):


\noindent 
\begin{tm}[Ellis]\quad 
Every compact Hausdorff left topological semigroup has an idempotent.
\end{tm}

(Earlier A.~Wallace and K.~Numakura stated independently this result 
for {\it topological\/} semigroups. However just one sided continuity 
is important for applications, as we'll see later.)

In particular, {\it any minimal compact left topological semigroup
consists of a unique element\/}.
Note that this statement has a (trivial) {\it purely algebraic\/} counterpart:
{\it Any minimal semigroup consists of a unique element\/}.
Here a groupoid is {\it minimal\/} iff it includes no proper subgroupoids,
and {\it minimal compact\/} iff it is compact and includes no proper compact subgroupoids.

\newpage

This result, interesting in itself, became crucial for various applications 
(of Ramsey-theoretic character) in number theory, algebra, topological dynamics, and ergodic theory.
These applications are based on using of {\it idempotent ultrafilters\/}.
(For a~general reference, see~[2].)
Perhaps, the simplest illustration of such applications is 
Hindman's famous {\it Finite Sums Theorem\/} (1974):

\noindent
\begin{tm}[Hindman]\quad
Any finite partition of~$\mathbb N$, the set of natural numbers,
has a~part that is ``big" in the following sense: it contains an infinite 
sequence all finite sums of distinct elements of which belong to this part.
\end{tm}

(Previously weaker results were given by various authors, 
including D.~Hilbert and I.~Schur).

The original purely combinatorial proof was uncredibly cumbersome.
(Recently N.~Hindman said: ``I~never understood the original complicated proof
(no, I did not plagiarize it)...".) 
Soon after this, however, F.~Galvin and S.~Glazer 
found a~way to obtain this result in a~few lines. 
To explain the idea, we need some preparations.

\newpage

{\it The space of ultrafilters.\/}\quad
Let $X$~be a~set.
Recall that the set $\scc X$ of all ultrafilters over~$X$ 
carries a natural topology with elementary (cl)open sets
$$O_A=\{u\in\scc X:A\in u\}$$ for each $A\subseteq X$.
The following facts are standard in general topology:
{\it $\scc X$ is a~compact Hausdorff extremally disconnected space. 
Moreover, it is the Stone--\v{C}ech compactification
of the discrete space~$X$.\/}

Recall that the {\it Stone--\v{C}ech compactification\/} of a~space~$X$ 
is a~compact space~$Z$ such that $X$~is dense in~$Z$, and 
for every compact space~$Y$, any continuous mapping $f:X\to Y$ 
can be extended to a~continuous mapping $\tilde f:Z\to Y$. 
Such a~$Z$ is unique up to homeomorphism with fixed points of~$X$.
It is customary to identify elements of~$X$ with principal ultrafilters, 
while non-principal ultrafilters form the {\it remainder\/}: 
$$X^*=\scc X\setminus X.$$
Under this identification, any unary operation on a~discrete space~$X$
can be extended to an operation on the space of ultrafilters over~$X$.

\newpage

{\it The algebra of ultrafilters.\/}\quad
Consider now extensions of binary operations. 
(In principle, the argument works for arbitrary operations.)

Let $(X,\cdot)$ be a~groupoid.
To extend $\cdot$ on~$X$ to $\cdot$ on~$\scc X$ 
(which let me denote by the same symbol), we proceed as follows:
First we extend $\cdot$ by fixing each second argument,
then by fixing each first argument.

We can define this extension in a~straightforward way by putting
$$
uv=
\{S\subseteq X:
\{a\in X:
\{b\in X:
ab\in S
\}\in u
\}\in v
\}
$$
for all $u,v\in\scc X$.
This looks slightly complicated but has the same clear meaning.

As a~result, the extended operation is continuous for any fixed first argument, 
i.e. the groupoid $(\scc X,\cdot)$ is {\it left\/} topological.
And although it is never {\it right\/} topological, 
except for trivial cases, the operation is continuous 
for any fixed second argument {\it whenever it is in~$X$\/}.
Moreover, it is a~unique operation with these properties.

\newpage

Not many algebraic properties are stable under this extension.

E.g. consider the extension of $(\mathbb N,+,\cdot)$, 
the semiring of natural numbers with the usual addition and multiplication. 
In $(\scc\mathbb N,+,\cdot)$, none of the extended operations is commutative.
Also both distributivity laws fail (I'll return to this example below).

However, the associativity law {\it is\/} stable:

\noindent
\begin{lm}\quad
If $X$~is a~semigroup, so is~$\scc X$.
\end{lm}

So any semigroup~$X$ extends to the compact left topological 
semigroup~$\scc X$ of ultrafilters over~$X$.
By Ellis' theorem, the latter has an idempotent.

Of course, it is in~$X^*$ if $X$~has no idempotents.
Another case when one can get an idempotent in~$X^*$ is 
when $X$~is (weakly) cancellative; in this case, 
$X^*$~forms a~subsemigroup of~$\scc X$, and 
as $X^*$~is closed in~$\scc X$ and so compact,
one can apply Ellis' theorem to $X^*$.

\newpage

After this preparations, the Finite Sums Theorem 
is reduced to the following fact:

\noindent
\begin{lm}\quad
Idempotents of $(\mathbb N^*,+)$ consist just of sets 
``big" in sense of that theorem.
\end{lm}

The proof of this lemma is not hard.

Actually, the lemma  is true not only for $(\mathbb N,+)$ but for 
any groupoid $(X,\cdot)$ {\it whenever $(X^*,\cdot)$ has idempotents\/}.
But to establish their existence we use associativity and apply Ellis' theorem.

On the other hand, no specific properties of semigroups are used,
thus the argument works for {\it any\/} semigroup. 
E.g. one can apply it to $(\mathbb N,\cdot)$ and prove 
the {\it Finite Products Theorem\/} (formulated analogously).

\newpage

Hindman's theorem, even in its general form, i.e. for 
every semigroups, is the {\it simplest\/} illustration here~--- since 
its proof uses {\it arbitrary\/} idempotent ultrafilters.
By using idempotent ultrafilters {\it of specific form\/}, 
one can prove a~lot of 
other significant results.
Most popular examples include:
\\
$$ 
\begin{array}{ll}
\;\text{van der Waerden's and Szemer\'edi's}
\\
\;\text{Arithmetic Progressions Theorem,}
\\
\\
\;\text{Furstenberg's Common Recurrence Theorem,} 
\\
\\
\;\text{Hales--Jewett's theorem,} 
\end{array}
$$
\\
etc.
Some of these results can be proved elementarily
but using of ultrafilters much simplifies this;
for other results, no elementary proofs are known.


\newpage
{\centerline{\Large{\bf Two operations:\/}}}
\noindent
{\centerline{\Large{\bf Left semirings\/}}}

\newpage

Here I~establish the main result of this talk,
the existence of common idempotents in compact left topological left semirings.
Also I~give some its application and consider its algebraic version.

\newpage

{\it Terminology.\/}\quad
An {\it algebra\/} means here a~universal algebra.
I~shall consider algebras with two binary operations.
Let me denote such an algebra by $(X,+,\cdot)$, or simply~$X$ again,
and refer to its operations as its {\it addition\/} and {\it multilication\/}.
(However, no their properties, like commutativity, associativity, etc.,
are assumed {\it a~priori\/}.)

Given an algebra $(X,+,\cdot)$, recall that
the multiplication is {\it left distributive w.r.t.\/}~the addition
iff the algebra satisfies the law
$$
x(y+z)=xy+xz.
$$
An algebra $(X,+,\cdot)$ is a~{\it left semiring\/} iff 
both its groupoids are semigroups and $\cdot$~is left distributive w.r.t.~$+$.
{\it Right distributivity\/} and {\it right semirings\/} have the dual definitions;
{\it semirings\/} are algebras which are left and right semirings simultaneously.

An algebra with binary operations is {\it left topological\/} 
iff any of its groupoids is left topological.

\newpage

Now we are ready to establish our main result:

\noindent
\begin{tm}\quad
Any compact Hausdorff left topological left semiring has a~common idempotent 
(and so consists of a~unique element whenever is minimal compact).
\end{tm}

Here a~{\it common idempotent\/} means an element that is 
an idempotent for each operations, in other words,
an element forming a~subalgebra.

This theorem {\it generalizes\/} Ellis' result on semigroups~--- since 
any left topological semigroup can be turned  
into a~left topological semiring in a~trivial way. E.g.: 

If $(X,\cdot)$ is a~left topological semigroup, 
let $+$~be the projection onto the {\it first\/} argument,
then \\$(X,+,\cdot)$ is a~left topological semiring. Or else:

If $(X,+)$ is a~left topological semigroup, 
let $\cdot$~be the projection onto the {\it second\/} argument,
then $(X,+,\cdot)$ is a~left topological semiring, too.

\newpage

\newpage

\begin{proof}\quad 
Let $(X,+,\cdot)$ be a~compact Hausdorff left topological left semiring.

1. 
First we show that {\it if $X$~is minimal compact,
then it consists of a~unique element\/}.

Let $a\in X$.
The set $aX=\{ax:x\in X\}$ is compact 
(as the image of~$X$ under $x\mapsto ax$).
Moreover, $aX$~forms a~subalgebra.
Therefore, $aX=X$ (by minimality).

Furthermore, 
the set $A=\{x\in X:ax=a\}$ is nonempty (since $aX=X$) 
and compact (as the preimage of~$\{a\}$ under $x\mapsto ax$).
Does $A$~form a~subalgebra?
In general, {\it no\/}: $(A,\cdot)$ is a~semigroup
but $(A,+)$ should not be a~semigroup.

But assume that $a$~is an additive idempotent; 
it exists by Ellis' theorem applied to $(X,+)$.
Then $(A,+,\cdot)$~{\it is\/} a~subalgebra.
Therefore, $A=X$ (by minimality), and thus
$a$~is also a~multiplicative idempotent.
It follows $X=\{a\}$ (by minimality again).

\newpage

2.
Now we consider the general case, when $X$~is not assumed minimal compact,
and show that it has a~common idempotent. 

The intersection of any $\subseteq$-decreasing 
chain of compact subalgebras of~$X$ is a~compact subalgebra.
By Zorn's Lemma, there is a~{\it minimal compact\/} subalgebra~$A$.
By the first part of the proof, it consists of a~unique element,
which is thus a~common idempotent.
\end{proof}


\noindent
\begin{rmk}\quad
The proof shows a stronger fact: {\it Any additive idempotent of every 
minimal compact left topological left semiring is also a~multiplicative one.\/}
Simple examples show that multiplicative idempotents should not be additive, and  
additive idempotents of {\it non\/}{}minimal compact left topological left semiring 
should not be multiplicative.
\end{rmk}

\newpage

{\it Algebraic variant.\/}\quad
As we noted, Ellis' result on minimal compact semigroups has 
an obvious {\it purely algebraic\/} counterpart: 
{\it Every minimal semigroup consists of a~unique element\/}.
One can ask whether the result on minimal compact semirings 
have the similar algebraic version:

\noindent
\begin{q}\quad
Can a~minimal 
left semiring have more than one element?
\end{q}

Although the question looks not difficult, 
I~was able to get the (expected) negative answer only in partial cases:

\noindent
\begin{pr}
\noindent
\\
1. Any minimal finite left semiring consists of a~unique element.
\\
2. Any minimal semirings consists of a~unique element. 
\end{pr}

Clause~1 follows from Theorem (consider the discrete topology),
while the proof of Clause~2 uses different arguments.
It remains to exclude a~possibility of a~minimal {\it countable\/} left semiring.

\newpage

{\it An application.\/}\quad 
We use the theorem to partially answer a~long-standing question
on $(\mathbb N^*,+,\cdot)$, 
the algebra of non-principal ultrafilters over~$\mathbb N$
with their (extended) addition and multiplication:

\noindent 
\begin{q}\quad
Can non-principal ultrafilters over~$\mathbb N$ 
be instances of left or right distributivity?
i.e. can some $u,v,w\in\mathbb N^*$ satisfy
$u(v+w)=uv+uw$ or $(u+v)w=uw+vw$?
\end{q}

E.~van Douwen proved that such instances, even if exist, 
are topologically rare (see~[3]):

\noindent
\begin{tm}[van Douwen]\quad
The sets
$$
\{u\in\mathbb N^*:
\forall v,w\in\mathbb N^*\; u(v+w)=uv+uw\}
$$
and 
$$
\{u\in\mathbb N^*:
\forall v,w\in\mathbb N^*\; (u+v)w=uw+vw\}
$$
are nowhere dense in~$\mathbb N^*$.
\end{tm}

We produce a~negative result in another direction:

\newpage

A~well-known (and difficult) result says that the algebra 
$(\mathbb N^*,+,\cdot)$ has no common idempotents.
It follows

\noindent
\begin{cor}\quad
No closed subalgebras of $(\mathbb N^*,+,\cdot)$
satisfy the left distributivity law.
\end{cor}

In this in mind, we may specify the question as follows:

\noindent 
\begin{q}\quad
Can some {\it non-closed\/} subalgebra of 
$(\mathbb N^*,+,\cdot)$ be a~left semiring?
\end{q}

\noindent 
\begin{rmk}\quad
In fact, $\mathbb N^*$~does not have even $u$ with $u+u=uu$. 
And it is open if some $u,v,w,x\in\mathbb N^*$ satisfy $u+v=wx$ (or other linear equations).
On the other hand, the closure of the set of additive idempotents forms 
a~right ideal of~$(\mathbb N^*,\cdot)$, so (by Ellis' result) 
there must be multiplicative idempotents {\it close\/} to additive ones.
This allows to refine Hindman's theorem:
{\it For any finite partition of~$\mathbb N$
there is a~part containing all finite sums of an infinite sequence 
as well as all finite products of another one.\/} 
(This refinement was proved by Hindman and Bergelson, see~[2].)
\end{rmk}

\newpage
{\centerline{\Large{\bf Generalizations:\/}}}
\noindent
{\centerline{\Large{\bf Non-associative case\/}}}

\newpage

Here we note that the used arguments work in more general situations,
when the operations are not assumed associative 
but satisfy certain other, much weaker conditions.

For simplicity, let me consider only the case with one operation.
Results for algebras with two (and in fact, with any number of) operations
can be established along the same line as this was done above for left semirings
(see~[4]).

\newpage

{\it Terminology.\/}\quad
An occurence of a~variable~$x$ into 
a~term $t(x,\ldots)$ is {\it right-most\/}
iff 
the occurence is non-dummy and
whenever
$$t_1(x,\ldots)\cdot t_2(x,\ldots)$$
is a~subterm of~$t$, 
then $x$~occurs non-dummy into 
$t_2(x,\ldots)$ but not~$t_1(x,\ldots)$.

\noindent
\begin{exms}\quad
All the occurences of the variable~$x$ into the terms
$$
x,\;\;
vx,\;\;
(vv)x,\;\;
v(vx),\;\;
(uv)(wx)
$$
are right-most,
while all its occurences into the terms
$$
v,\;\;
uv,\;\;
xv,\;\;
xx,\;\;
x(vx)\;\;
(ux)(vx)
$$
are not.
\end{exms}

A~{\it left-most\/} occurence can be defined dually
but is not used in what follows.
Note that if the occurence of~$x$ into~$t$ is right-most 
(or left-most), then $x$~occurs into~$t$ just one time.

\newpage

\begin{lm}\quad
Let $X$~be a~left topological groupoid,
and let $t(v,\ldots,x)$ be a~term
with the right-most occurence of the last argument.
Then for every $a,\ldots\in X$,
the mapping
$$x\mapsto t(a,\ldots,x)$$
is continuous.
\end{lm}

\begin{proof}
By induction on construction of~$t$.
\end{proof}

Now we are ready to formulate our non-associative version of Ellis' result:

\newpage

\begin{tm}\quad
Let $X$~be a~compact left topological groupoid,
and $r$, $s$, and $t$~some one-, two-, and three-parameter terms respectively,
where $s$~has the right-most occurence of the last argument.
Suppose that $X$~satisfies
$$
s(x,y)=s(x,z)=r(x)
\;\to\;
s(x,y\cdot z)=r(x)
$$
and
$$
s(x,y)\cdot 
s(x,z)
=
s(x,t(x,y,z)).
$$
Then $X$~has an idempotent.
\end{tm}

The proof refines the argument used for semigroups.
First one shows that if such a~groupoid is {\it minimal compact\/}, 
then it consists of a~unique element (the lemma above is necessary here). 
Then one applies Zorn's Lemma 
to isolate a~minimal compact subgroupoid (which satisfies 
the same universal formulas).


The theorem indeed generalizes Ellis' one~--- since
associativity implies the conditions of this theorem,
with
$$
r(x)=x,\quad
s(x,y)=xy,\quad
t(x,y,z)=yxz.
$$
But many other {\it identities\/} imply these conditions as well. 
Let me give several examples.

\newpage

{\it Examples.\/}
1. 
Consider four following identities (``in Moufang style"):
$$
\begin{array}{lccr}
(1)
&\qquad\quad&
(xy)(yz)=((xy)y)z,
&\qquad\qquad
\\
(2)
&\qquad\quad&
(xz)(yz)=((xz)y)z,
&\qquad\qquad
\\
(3)
&\qquad\quad&
(xy)(xz)=x(y(xz)),
&\qquad\qquad
\\
(4)
&\qquad\quad&
(xx)(yz)=((xx)y)z
&\qquad\qquad
\end{array}
$$
Each of them is {\it strictly weaker\/} than associativity.

(1),~as well as~(2), implies the first condition of the theorem,
while (3)~implies the second one, all with $r(x)=x$ and $s(x,y)=xy$ again.
(4)~implies the first condition with $r(x)=x$ and $s(x,y)=(xx)y$.
Therefore, {\it every compact left topological groupoid satisfying 
any of (1), (2), (4), togheter with~(3), has an idempotent.\/}


2. 
The identity 
$$
\:(5)
\qquad\qquad\quad 
x(yz)=(xz)y
\qquad\qquad\qquad\quad\; 
$$ 
is {\it incomparable\/} with associativity but, 
like it, implies both conditions of the theorem, again with $r(x)=x$ and $s(x,y)=xy$.
Hence, {\it every compact left topological groupoid satisfying~(5) has an idempotent.\/}

\newpage

3.
The {\it left self-distributivity\/} 
$$
\,(6)
\qquad\quad\quad 
x(yz)=(xy)(xz)
\qquad\qquad\qquad\quad 
$$ 
appeared in the algebra of elementary embeddings
arising under (very) large cardinal hypotheses (R.~Laver). 
Also, P.~Dehornoy found its deep connection with infinite braid groups.
(See~[5].)
It is not hard to check that the conjunction of~(6) with the identity 
$$
\,(7)
\qquad\qquad\quad 
x(xx)=(xx)x
\qquad\qquad\qquad\quad 
$$ 
implies both conditions of the theorem, 
with \\$r(x)=x(xx)$, $s(x,y)=xy$, $t(x,y,z)=yz$. 
Hence, {\it any compact left topological groupoid satisfying 
(6) and~(7) has an idempotent.\/}

(The identity~(7) is necessary; left self-distributivity alone implies
only {\it finiteness\/} of minimal groupoids.)

\newpage

Our non-associative theorem can be interesting from the following point:

\noindent
\begin{q}\quad
Which formulas (or, at least, which identities) are stable
under passing to the algebra of ultrafilters?
\end{q}

\noindent
\begin{cnj}\quad
Identities that follows from associativity are stable
under passing to the algebra of ultrafilters.
\end{cnj}

E.g. so are the identities (1)--(4) above.

{\it Observation.\/}\quad 
If some identities, on the one hand, imply the conditions of the theorem,
and, on the other hand, are stable under passing to ultrafilters, 
then the technique based on ultrafilters and developed for semigroups
can be extended to groupoids satisfying these identities.
E.g. one can prove for such groupoids (appropriate analogs of) 
Hindman's theorem, van der Waerden's theorem, etc.

\newpage

\newpage

{\centerline{\Large{\bf References\/}}}

\newpage

[1]\; 
Ellis~R.\;
{\it Lectures on topological dynamics.\/}\;
\\ Benjamin, NY, 1969.

[2]\; 
Hindman~N.,\; Strauss~D.\;\;
{\it Algebra\, in\, the \\ Stone--\v{C}ech\, compactification.\/}\,
W.~de~Gruyter, Berlin--NY, 1998.

[3]\; 
van~Douwen~E.\; 
{\it The Stone--\v{C}ech compactification 
of a~discrete groupoid.\/}\;
Topol.\/ Appl., 39~(1991), 43--60.

[4]\; 
Saveliev~D.~I.\;
{\it Idempotents in compact left topological algebras 
with binary operations.\/}\;
To appear.

[5]\; 
Dehornoy~P.\;
{\it Elementary embeddings and \\ algebra.\/}
In: Foreman~M.,\; Kanamori~A. (eds.),
\\ {\it Handbook of Set Theory\/}.
To appear.

\end{document}